\documentclass[12pt]{article}

\usepackage{amsmath, epsfig, cite, setspace}
\usepackage{amssymb}
\usepackage{amsfonts}
\usepackage{latexsym}
\usepackage{amsthm}

\makeatletter
\renewcommand{\@seccntformat}[1]{{\csname the#1\endcsname}.\hspace{.5em}}
\makeatother

\newtheorem{thm}{Theorem}[section]

\newtheorem{conj}[thm]{Conjecture}
\newtheorem{lem}[thm]{Lemma}
\newtheorem{remark}[thm]{Remark}

\renewcommand{\qed}{\hfill$\Box$\medskip}
\renewcommand{\thefootnote}{*}

\numberwithin{equation}{section}

\begin{document}

\begin{center}
{\large\bf Proof of some supercongruences via the Wilf-Zeilberger method}
\end{center}

\vskip 2mm \centerline{Guo-Shuai Mao}
\begin{center}
{\footnotesize $^1$Department of Mathematics, Nanjing
University of Information Science and Technology, Nanjing 210044,  People's Republic of China\\
{\tt maogsmath@163.com  } }
\end{center}


\vskip 0.7cm \noindent{\bf Abstract.}
In this paper, we prove some supercongruences via the Wilf-Zeilberger method. For instance, for any odd prime $p$ and positive integer $r$ and $\delta\in\{1,2\}$, we have
\begin{align*}
\sum_{n=0}^{(p^r-1)/\delta}\frac{\left(\frac12\right)^5_n}{n!^5}(10n^2+6n+1)(-4)^n&\equiv\begin{cases}p^{2r}\ \pmod{p^{r+4}} &\tt{if}\ r\leq4,
\\0\ \pmod{p^{r+4}} &\tt{if}\ r\geq5.\end{cases}
\end{align*}
\vskip 3mm \noindent {\it Keywords}: Supercongruences; binomial coeficients; Wilf-Zeilberger method.

\vskip 0.2cm \noindent{\it AMS Subject Classifications:} 11A07, 05A10.

\renewcommand{\thefootnote}{**}

\section{Introduction}
In the past decade, many researchers studied supercongruences via the Wilf-Zeilberger (WZ) method. For instance, W. Zudilin \cite{zudilin-jnt-2009} proved several Ramanujan-type supercongruences by the WZ method. One of them, conjectured by van Hamme, says that
\begin{align}\label{wzpr}
\sum_{k=0}^{(p-1)/2}(4k+1)(-1)^k\left(\frac{\left(\frac12\right)_k}{k!}\right)^3\equiv(-1)^{(p-1)/2}p\pmod{p^3},
\end{align}
where $(a)_n=a(a+1)\ldots(a+n-1) (n\in\{1,2,\ldots\})$ with $(a)_0=1$ is the raising factorial for $a\in\mathbb{C}$. He \cite{H-pams-2015} also obtained some supercongruences modulo $p^4$.

For $n\in\mathbb{N}$, define
$$H_n:=\sum_{0<k\leq n}\frac1k, H_0=0.$$
This $H_n$ with $n\in\mathbb{N}$ are the classical harmonic numbers. Let $p>3$ be a prime. J. Wolstenholme \cite{wolstenholme-qjpam-1862} proved that
$$H_{p-1}\equiv0\pmod{p^2}\ \mbox{and}\ H_{p-1}^{(2)}\equiv0\pmod p,$$
which imply that
\begin{align}
\binom{2p-1}{p-1}\equiv1\pmod{p^3}.\label{2p1p}
\end{align}
Throughout the paper, $p$ is an odd prime and $r$ is a positive integer. Guillera and Zudilin \cite{GZ-pams-2012} proved that
$$
\sum_{n=0}^{(p-1)/2}\frac{\left(\frac12\right)^5_n}{n!^5}(10n^2+6n+1)(-4)^n\equiv p^2\pmod{p^5}.
$$
We should generalize their result to the following form:
\begin{thm}\label{Th10ngz} For $\delta\in\{1,2\}$, we have
\begin{align}\label{5n1}
\sum_{n=0}^{(p^r-1)/\delta}\frac{\left(\frac12\right)^5_n}{n!^5}(10n^2+6n+1)(-4)^n&\equiv\begin{cases}p^{2r}\ \pmod{p^{r+4}} &\tt{if}\ r\leq4,
\\0\ \pmod{p^{r+4}} &\tt{if}\ r\geq5.\end{cases}
\end{align}
\end{thm}
Actually, we have the following conjecture which cannot be proved by our method:
\begin{conj}
\begin{equation*}
\sum_{n=0}^{(p^r-1)/\delta}\frac{\left(\frac12\right)^5_n}{n!^5}(10n^2+6n+1)(-4)^n\equiv p^{2r}\pmod{p^{2r+3}}.
\end{equation*}
\end{conj}
Zudilin \cite{zudilin-jnt-2009} also proved that
\begin{align}\label{20n33}
\sum_{n=0}^{p-1}\frac{\left(\frac12\right)_n\left(\frac12\right)_{2n}}{n!^3}\frac{20n+3}{2^{4n}}\equiv3 p(-1)^{(p-1)/2}\pmod{p^3},
\end{align}
\begin{align}\label{120n2}
\sum_{n=0}^{p-1}\frac{\left(\frac12\right)^3_n\left(\frac12\right)_{2n}}{n!^5}\frac{120n+34n+3}{2^{6n}}\equiv3p^2\pmod{p^5}.
\end{align}
We generalize (\ref{20n33}) to the following form:
\begin{thm}\label{Th20n3}
$$
\sum_{n=0}^{p^r-1}\frac{\left(\frac12\right)_n\left(\frac12\right)_{2n}}{n!^3}\frac{20n+3}{2^{4n}}\equiv3(-1)^{(p^r-1)/2}p^r\pmod{p^{r+2}}.
$$
\end{thm}
\begin{remark} \rm Actually, we also can obtain the following result with $\delta\in\{1,2\}$ which generalizes (\ref{120n2}), here we won't prove it since the proof is similar to that of Theorem \ref{Th10ngz}.
\begin{align*}
\sum_{n=0}^{(p^r-1)/\delta}\frac{\left(\frac12\right)^3_n\left(\frac12\right)_{2n}}{n!^5}\frac{(120n^2+34n+3)}{2^{6n}}&\equiv\begin{cases}p^{2r}\ \pmod{p^{r+4}} &\tt{if}\ r\leq4,
\\0\ \pmod{p^{r+4}} &\tt{if}\ r\geq5.\end{cases}
\end{align*}
\end{remark}
Guo \cite{g-jmaa-2018} proved that
$$
\sum_{k=0}^{(p^r-1)/2}\frac{4k+1}{(-64)^k}\binom{2k}k^3\equiv(-1)^{\frac{(p-1)r}2}p^r\pmod{p^{r+2}},
$$
and in the same paper he proposed a conjecture as follow:
\begin{conj}\label{Guo-2018}{\rm (\cite[Conjecture 5.1]{g-jmaa-2018})}
$$
\sum_{k=0}^{p^r-1}\frac{4k+1}{(-64)^k}\binom{2k}k^3\equiv(-1)^{\frac{(p-1)r}2}p^r\pmod{p^{r+2}}.
$$
\end{conj}
Guo and zudilin have proved Conjecture \ref{Guo-2018} by founding its $q$-anology, (see \cite{GZ-aim-2019}), here we give a new proof of it by the WZ method. Our way differs from their because we used the result $-2p^r/(k\binom{2k}k)\equiv\binom{2p^r-2k}{p^r-k}\pmod{p^2}$ for each $1\leq k\leq(p^r-1)/2$ which was in \cite{PS}, and we also used a result of Sun \cite{sun-jnt-2011}, $\sum_{k=0}^{(p-3)/2}\frac{\binom{2k}k}{(2k+1)4^k}\equiv-(-1)^{(p-1)/2}q_p(2)\pmod{p^2}.$

Z.-W. Sun \cite{sun-ijm-2012} proved the following congruence by the WZ method
\begin{equation}\label{sun}
\sum_{k=0}^{p-1}\frac{4k+1}{(-64)^k}\binom{2k}k^3\equiv(-1)^{\frac{(p-1)}2}p+p^3E_{p-3}\pmod{p^4}.
\end{equation}
In this paper we first prove the above conjecture.
\begin{thm}\label{Thgjmaa}
Conjecture \ref{Guo-2018} is true.
\end{thm}
Guo and Liu \cite{gl-arxiv-2019} showed that
\begin{equation}\label{glp4}
\sum_{k=0}^{(p+1)/2}(-1)^k(4k-1)\frac{\left(-\frac12\right)_k^3}{(1)_k^3}\equiv p(-1)^{(p+1)/2}+p^3(2-E_{p-3})\pmod{p^4},
\end{equation}
where $E_n$ are the Euler numbers defined by
$$E_0=1,\ \mbox{and}\ E_n=-\sum_{k=1}^{\lfloor n/2\rfloor}\binom{n}{2k}E_{n-2k}\ \mbox{for}\ n\in\{1,2,\ldots\}.$$
They also gave some conjectures in the last section of \cite{gl-arxiv-2019}. For instance,
\begin{conj}\label{Guo-Liu} {\rm (\cite[Conjecture 5.1]{gl-arxiv-2019})}
\begin{equation*}
\sum_{k=0}^{p^r-1}(-1)^k(4k-1)\frac{\left(-\frac12\right)_k^3}{(1)_k^3}\equiv-(-1)^{\frac{(p-1)r}2}p^r\pmod{p^{r+2}}.
\end{equation*}
\end{conj}
Guo has proved this Conjecture by founding its $q$-analogy, (see \cite{guo-rama-2019}). Here we also give a new proof of this conjecture by the WZ method which differs from Guo's method. Now we list our second result.
\begin{thm}\label{Thglarxiv2}
Conjecture \ref{Guo-Liu} is true.
\end{thm}
Via an identity in \cite[Lemma 2.2]{mao-rama-2018}, we generalize congruence $(I.2)$ of van Hamme which also can be found in \cite{swisher}.
\begin{thm}\label{swisher}
\begin{equation}\label{gvan}
\sum_{n=0}^{(p-1)/2}\frac{\left(\frac12\right)_n^2}{(n+1)n!^2}\equiv2p^2+2p^3(2q_p(2)-1)\pmod{p^4},
\end{equation}
where $q_p(2)$ denotes the Fermat quotient $(2^{p-1}-1)/p$.
\end{thm}
Our main tool is the WZ method. We shall prove Theorem \ref{Th10ngz} in Section 2, Theorems \ref{Thgjmaa} and \ref{Thglarxiv2} will be proved in Sections 3 and 4, respectively. And Theorem \ref{Th20n3} will be proved in Section 5. The last Section is devoted to prove Theorem \ref{swisher}.
\section{Proof of Theorem \ref{Th10ngz}}
First we have the following WZ pair (about the WZ method, see, for instance, \cite{AZ, PWZ, Z}) in \cite{GZ-pams-2012}
$$
F(n,k)=(10n^2+12nk+6n+4k^2+4k+1)\frac{\left(\frac12\right)_n\left(\frac12+k\right)^4_n}{(1)^5_n}(-1)^n2^{2n}
$$
and
$$
G(n,k)=(n+2k-1)\frac{\left(\frac12\right)_n\left(\frac12+k\right)^4_{n-1}}{(1)^5_{n-1}}(-1)^n2^{2n+1}.
$$
It is easy to check that
\begin{align}\label{FG123}
F(n,k-1)-F(n,k)=G(n+1,k)-G(n,k).
\end{align}
Summing up the above equation for $n$ from $0$ to $(p^r-1)/2$, and then for $k$ from $1$ to $(p^r-1)/2$, we get
\begin{equation}\label{wzfin}
\sum_{n=0}^{(p^r-1)/2}F(n,0)=\sum_{n=0}^{(p^r-1)/2}F(n,(p^r-1)/2)+\sum_{k=1}^{(p^r-1)/2}G((p^r+1)/2,k).
\end{equation}
\begin{lem}\label{Lemwzfin1} For $\delta\in\{1,2\}$, we have
$$
\sum_{n=0}^{(p^r-1)/\delta}F(n,(p^r-1)/2)\equiv p^{2r} \pmod{p^{2r+3}}.
$$
\end{lem}
\begin{proof}
By the definition of $F(n,k)$, we have
$$F(0,(p^r-1)/2)=p^{2r}$$
and
\begin{align*}
\sum_{n=1}^{(p^r-1)/\delta}F(n,(p^r-1)/2)&=\sum_{n=1}^{(p^r-1)/\delta}(10n^2+6np^r+p^{2r})\frac{\left(\frac12\right)_n\left(\frac{p^r}2\right)^4_n}{(1)^5_n}(-4)^n\\
&=\frac{p^{4r}}{16}\sum_{n=1}^{(p^r-1)/\delta}(10n^2+6np^r+p^{2r})\frac{(-1)^n\binom{2n}n}{n^4}\binom{p^r/2+n-1}{n-1}^4.
\end{align*}
It is easy to see that $p^{5r}/n^3\equiv p^{6r}/n^4\equiv0\pmod{p^{2r+3}}$ for each $1\leq n\leq(p^r-1)/\delta$. So
\begin{align*}
\sum_{n=1}^{(p^r-1)/\delta}F(n,(p^r-1)/2)&\equiv\frac{5p^{4r}}{8}\sum_{n=1}^{(p^r-1)/\delta}\frac{(-1)^n\binom{2n}n}{n^2}\binom{p^r/2+n-1}{n-1}^4\\
&\equiv\frac{5p^{4r}}{8}\sum_{n=1}^{(p^r-1)/\delta}\frac{(-1)^n\binom{2n}n}{n^2}\equiv\frac{5p^{2r+2}}{8}\sum_{k=1}^{(p-1)/\delta}\frac{(-1)^k\binom{2p^{r-1}k}{p^{r-1}k}}{k^2}\\
&\equiv\frac{5p^{2r+2}}{8}\sum_{k=1}^{(p-1)/\delta}\frac{(-1)^k\binom{2k}{k}}{k^2}\pmod{p^{2r+3}}
\end{align*}
with $p^{4r}/n^2\equiv0\pmod{p^{2r+2}}$, $\binom{p^r/2+n-1}{n-1}\equiv 1\pmod p$ and Lucas congruence.

Therefore we complete the proof of Lemma \ref{Lemwzfin1} with \cite[(14)]{GZ-pams-2012}.
\end{proof}
\begin{lem}\label{Lemwzfin2}
$$\sum_{k=1}^{(p^r-1)/2}G((p^r+1)/2,k)\equiv0\pmod{p^{r+4}}.$$
\end{lem}
\begin{proof}
By the definition of $G(n,k)$ we have
\begin{align*}
&\sum_{k=1}^{(p^r-1)/2}G((p^r+1)/2,k)=\sum_{k=1}^{(p^r-1)/2}(p^r-1+2k)\frac{\left(\frac12\right)_{(p^r+1)/2}\left(\frac12+k\right)^4_{(p^r-1)/2}}{(1)^5_{(p^r-1)/2}}(-4)^{(p^r+1)/2}\\
&=-2p^r\binom{p^r-1}{(p^r-1)/2}(-1)^{(p^r-1)/2}\sum_{k=1}^{(p^r-1)/2}(p^r-1+2k)\frac{\left(\frac12+k\right)^4_{(p^r-1)/2}}{(1)^4_{(p^r-1)/2}}\\
&=-2p^r\binom{p^r-1}{(p^r-1)/2}^3(-1)^{(p^r-1)/2}\sum_{k=1}^{(p^r-1)/2}(p^r-1+2k)\frac{\binom{p^r+2k-1}{(p^r-1)/2+k}^2\binom{p^r+2k-1}{2k}^2}{\binom{2k}k^2}\\
&=-\frac{p^{3r}}2\binom{p^r-1}{(p^r-1)/2}^3(-1)^{(p^r-1)/2}\sum_{k=1}^{(p^r-1)/2}(p^r-1+2k)\frac{\binom{p^r+2k-1}{(p^r-1)/2+k}^2\binom{p^r+2k-1}{2k-1}^2}{\binom{2k}k^2}.
\end{align*}
It is easy to see that
\begin{align*}
ord_p\left(\binom{p^r+2k-1}{(p^r-1)/2+k}\right)&=\sum_{j=1}^{\infty}\left(\left(\left\lfloor\frac{p^r-1+2k}{p^j}\right\rfloor\right)-2\left(\left\lfloor\frac{p^r-1+2k}{2p^j}\right\rfloor\right)\right)\\
&=\sum_{j=1}^{r}\left(\left(\left\lfloor\frac{p^r-1+2k}{p^j}\right\rfloor\right)-2\left(\left\lfloor\frac{p^r-1+2k}{2p^j}\right\rfloor\right)\right)\geq1
\end{align*}
since $\left(\left\lfloor\frac{p^r-1+2k}{p^r}\right\rfloor\right)-2\left(\left\lfloor\frac{p^r-1+2k}{2p^r}\right\rfloor\right)=1$ and for any real numbers $x$,
$$\lfloor2x\rfloor\geq 2\lfloor x\rfloor.$$
In view of the paper \cite{PS}, let $k$ and $l$ be positive integers with $k+l=p^r$ and $0<l<p^r/2$, we have
\begin{align}\label{2ll2kk}
l\binom{2l}l\binom{2k}k\equiv-2p^r\pmod{p^{r+1}},
\end{align}
\begin{align}\label{2kk}
\binom{2k}k\equiv0\pmod{p}
\end{align}
and
\begin{align}\label{dao2kk}
\frac{-2p^r}{l\binom{2l}l}\equiv\binom{2k}k\pmod{p^{2}}.
\end{align}
Hence with (\ref{dao2kk}) and (\ref{2kk}) we immediately obtain that
$$
\sum_{k=1}^{(p^r-1)/2}G((p^r+1)/2,k)\equiv0\pmod{p^{r+4}}.
$$
So we finish the proof of Lemma \ref{Lemwzfin2}.
\end{proof}
\noindent{\rm Case 1}. $\delta=2$. Substituting Lemmas \ref{Lemwzfin1} and \ref{Lemwzfin2} into (\ref{wzfin}), we immediately get the desired result.

 Summing up equation (\ref{FG123}) for $n$ from $0$ to $p^r-1$, and then for summing up $k$ from $1$ to $(p^r-1)/2$, we get
\begin{equation}\label{wzfin1}
\sum_{n=0}^{p^r-1}F(n,0)=\sum_{n=0}^{p^r-1}F(n,(p^r-1)/2)+\sum_{k=1}^{(p^r-1)/2}G(p^r,k).
\end{equation}
\begin{lem}\label{Lemwzfin3}
$$\sum_{k=1}^{(p^r-1)/2}G(p^r,k)\equiv0\pmod{p^{r+4}}.$$
\end{lem}
\begin{proof}
By the definition of $G(n,k)$ we have
\begin{align*}
\sum_{k=1}^{(p^r-1)/2}G(p^r,k)&=2\sum_{k=1}^{(p^r-1)/2}(p^r-1+2k)\frac{\left(\frac12\right)_{p^r}\left(\frac12+k\right)^4_{p^r-1}}{(1)^5_{p^r-1}}(-4)^{p^r}\\
&=-4p^r\binom{2p^r-1}{p^r-1}\sum_{k=1}^{(p^r-1)/2}(p^r-1+2k)\frac{\left(\frac12+k\right)^4_{p^r-1}}{(1)^4_{p^r-1}}\\
&=-4p^r\binom{2p^r-1}{p^r-1}\sum_{k=1}^{(p^r-1)/2}(p^r-1+2k)\frac{\binom{2p^r+2k-2}{p^r-1+k}^4\binom{p^r+k-1}{k}^4}{4^{4(p^r-1)}\binom{2k}k^4}\\
&=-4p^r\binom{2p^r-1}{p^r-1}\sum_{k=1}^{(p^r-1)/2}(p^r-1+2k)\frac{p^{4r}\binom{2p^r+2k-2}{p^r-1+k}^4\binom{p^r+k-1}{k-1}^4}{4^{4(p^r-1)}k^4\binom{2k}k^4}.
\end{align*}
Hence with (\ref{dao2kk}) and (\ref{2kk}) we immediately obtain that
$$
\sum_{k=1}^{(p^r-1)/2}G(p^r,k)\equiv0\pmod{p^{r+4}}.
$$
Now the proof of Lemma \ref{Lemwzfin3} is finished.
\end{proof}
\noindent{\rm Case 2}. $\delta=1$. Combining Lemmas \ref{Lemwzfin1} and \ref{Lemwzfin3} with (\ref{wzfin1}) we immediately obtain the result.

At this time, the proof of Theorem \ref{Th10ngz} is complete. \qed
\section{Proof of Theorem \ref{Thgjmaa}}
We will use the following WZ pair to prove Theorem \ref{Thgjmaa}. For nonnegative integers $n, k$, define
$$
F(n,k)=\frac{(-1)^{n+k}(4n+1)}{4^{3n-k}}\binom{2n}n^2\frac{\binom{2n+2k}{n+k}\binom{n+k}{2k}}{\binom{2k}k}
$$
and
$$
G(n,k)=\frac{(-1)^{n+k}(2n-1)^2\binom{2n-2}{n-1}^2}{2(n-k)4^{3(n-1)-k}}\binom{2(n-1+k)}{n-1+k}\frac{\binom{n-1+k}{2k}}{\binom{2k}k}.
$$
Clearly $F(n,k)=G(n,k)=0$ if $n<k$. It is easy to check that
\begin{equation}\label{FG}
F(n,k-1)-F(n,k)=G(n+1,k)-G(n,k)
\end{equation}
for all nonnegative integer $n$ and $k>0$.

We mentioned that Sun has proved the theorem for $r=1$, so we just need to show that for $r>1$.

Summing (\ref{FG}) over $n$ from $0$ to $p^r-1$ we have
$$
\sum_{n=0}^{p^r-1}F(n,k-1)-\sum_{n=0}^{p^r-1}F(n,k)=G(p^r,k)-G(0,k)=G(p^r,k).
$$
Furthermore, summing both side of the above identity over $k$ from $1$ to $p^r-1$, we obtain
\begin{align}\label{wz1}
\sum_{n=0}^{p^r-1}F(n,0)=F(p^r-1,p^r-1)+\sum_{k=1}^{p^r-1}G(p^r,k).
\end{align}
\begin{lem} \label{Fp1p1}
$$F(p^r-1,p^r-1)\equiv 0\pmod{p^{r+2}}.$$
\end{lem}
\begin{proof} Since $r>1$, we have
\begin{align*}
F(p^r-1,p^r-1)=\frac{(4p^r-3)}{4^{2p^r-2}}\binom{2p^r-2}{p^r-1}^2\frac{\binom{4p^r-4}{2p^r-2}}{\binom{2p^r-2}{p^r-1}}=\frac{p^{2r}\binom{2p^r-1}{p^r-1}\binom{4p^r-1}{2p^r-1}}{(4p^r-1)4^{2p^r-2}}\equiv0\pmod{p^{r+2}}.
\end{align*}
\end{proof}
By the definition of $G(n,k)$ we have
\begin{align}\label{Gpk}
G(p^r,k)&=\frac{(-1)^{k+1}(2p^r-1)^2\binom{2p^r-2}{p^r-1}^2}{2(p^r-k)4^{3(p^r-1)-k}}\binom{2p^r-2+2k}{p^r-1+k}\frac{\binom{p^r-1+k}{2k}}{\binom{2k}k}\notag\\
&=\frac{(-1)^{k+1}p^{2r}\binom{2p^r-1}{p^r-1}^2}{2(p^r-k)4^{3(p^r-1)-k}}\binom{2p^r-2+2k}{2k}\frac{\binom{2p^r-2}{p^r-1-k}}{\binom{2k}k}\notag\\
&=\frac{(-1)^{k+1}p^{2r}\binom{2p^r-1}{p^r-1}^2}{2(2p^r-1)4^{3(p^r-1)-k}}\binom{2p^r-2+2k}{2k}\frac{\binom{2p^r-1}{p^r-k}}{\binom{2k}k},
\end{align}
where we used the binomial transformation
$$
\binom{n}{k}\binom{k}{j}=\binom{n}{j}\binom{n-j}{k-j}.
$$
\begin{lem} \label{G12}
$$
\sum_{k=1}^{(p^r-1)/2}G(p^r,k)\equiv 0\pmod{p^{r+2}}.
$$
\end{lem}
\begin{proof}
By (\ref{Gpk}), we have
\begin{align*}
\sum_{k=1}^{(p^r-1)/2}G(p^r,k)=-\frac{p^{2r}\binom{2p^r-1}{p^r-1}^2}{2(2p^r-1)4^{3(p^r-1)}}\sum_{k=1}^{(p^r-1)/2}\frac{(-4)^k\binom{2p^r-2+2k}{2k}\binom{2p^r-1}{p^r-k}}{\binom{2k}k}.
\end{align*}
In view of (\ref{dao2kk}) we have the following congruence modulo $p^{r+2}$
\begin{align*}
\sum_{k=1}^{(p^r-1)/2}G(p^r,k)\equiv\frac{p^{r}\binom{2p^r-1}{p^r-1}^2}{4(2p^r-1)4^{3(p^r-1)}}\sum_{k=1}^{(p^r-1)/2}k(-4)^k\binom{2p^r-2+2k}{2k}\binom{2p^r-1}{p^r-k}\binom{2p^r-2k}{p^r-k}.
\end{align*}
It is easy to check that
\begin{align}\label{2pr2k}
\binom{2p^r-2+2k}{2k}&=\binom{-2p^r+1}{2k}=\frac{1-2p^r}{2k}\frac{-2p^r}{2k-1}\binom{-2p^r-1}{2k-2}\notag\\
&=\frac{p^r(2p^r-1)}{k(2k-1)}\binom{-2p^r-1}{2k-2}\equiv0\pmod p
\end{align}
since $\frac{p^r}{k(2k-1)}\equiv0\pmod p$ for all $1\leq k\leq p^r-1$ with $k\neq(p^r+1)/2$.

And with (\ref{2kk}) we have $\binom{2p^r-2k}{p^r-k}\equiv 0\pmod p$ for all $1\leq k\leq(p^r-1)/2$. Hence
$$\sum_{k=1}^{(p^r-1)/2}G(p^r,k)\equiv 0\pmod{p^{r+2}}.$$
\end{proof}
\begin{lem}\label{Gpr}
$$
G(p^r,(p^r+1)/2)\equiv(-1)^{(p^r-1)/2}p^r(1-3pq_p(2))\pmod{p^{r+2}},
$$
where $q_p(2)=(2^{p-1}-1)/p$ stands for the Fermat quotient.
\end{lem}
\begin{proof}
In view of (\ref{Gpk}), we have
\begin{align*}
G(p^r,(p^r+1)/2)&=\frac{(-1)^{(p^r-1)/2}p^{2r}\binom{2p^r-1}{p^r-1}^2}{2(2p^r-1)4^{3p^r-3-(p^r+1)/2}}\frac{\binom{3p^r-1}{p^r+1}\binom{2p^r-1}{(p^r-1)/2}}{\binom{p^r+1}{(p^r+1)/2}}\\
&=\frac{(-1)^{(p^r-1)/2}p^{r}(p^r+1)\binom{2p^r-1}{p^r-1}^2}{8(2p^r-1)4^{3p^r-3-(p^r+1)/2}}\frac{\binom{3p^r-1}{2p^r-2}\binom{2p^r-1}{(p^r-1)/2}}{\binom{p^r-1}{(p^r-1)/2}}.
\end{align*}
By \cite[Lemma 2.4]{long-2011-pjm} we have
\begin{align}\label{pr12pr}
\binom{p^r-1}{(p^r-1)/2}\equiv(-1)^{(p^r-1)/2}4^{p^r-1}\pmod{p^3}.
\end{align}
And it is easy to check that
\begin{align}\label{mao}
\frac{p^r+1}{2p^r-1}\binom{3p^r-1}{2p^r-2}&=\binom{3p^r-1}{2p^r-1}=\binom{3p^r-1}{p^r}=-\prod_{k=1}^{p^r}\left(1-\frac{3p^r}k\right)\notag\\
&=2\prod_{k=1}^{p^r-1}\left(1-\frac{3p^r}k\right)\equiv2(1-3p^rH_{p^r-1})\notag\\
&\equiv2(1-3pH_{p-1})\equiv2\pmod {p^2}
\end{align}
with Wolstenholme's result $H_{p-1}\equiv0\pmod{p^2}$ as we mentioned in the introduction.

In the same way, we have
\begin{align}\label{2prprp2}
\binom{2p^r-1}{p^r-1}\equiv1\pmod{p^2}
\end{align}
and
\begin{align}\label{2pr12pr}
\binom{2p^r-1}{(p^r-1)/2}\equiv(-1)^{(p^r-1)/2}(1-2pH_{(p-1)/2})\pmod{p^2}.
\end{align}
Hence
$$
G(p^r,(p^r+1)/2)\equiv\frac{(-1)^{(p^r-1)/2}p^{r}(1-2pH_{(p-1)/2})}{2^{7(p^r-1)}}\pmod{p^{r+2}}.
$$
Therefore the desired result immediately obtained since
\begin{align}\label{h12p}
H_{(p-1)/2}\equiv-2q_p(2)\pmod p
\end{align}
and
$$2^{p^r-1}\equiv1+pq_p(2)\pmod{p^2},$$
the congruence (\ref{h12p}) can be found in \cite{sunh-jnt-2008}.
\end{proof}
\begin{lem}\label{sun1}{\rm (\cite[(1.1)]{sun-jnt-2011})} Let $p$ be an odd prime. Then
$$
\sum_{k=0}^{(p-3)/2}\frac{\binom{2k}k}{(2k+1)4^k}\equiv-(-1)^{(p-1)/2}q_p(2)\pmod{p^2}.
$$
\end{lem}
\begin{lem}\label{Gpr12pr}
$$\sum_{k=(p^r+3)/2}^{p^r-1}G(p^r,k)\equiv(-1)^{(p^r-1)/2}3p^{r+1}q_p(2)\pmod{p^{r+2}}.$$
\end{lem}
\begin{proof}
Again by (\ref{Gpk}), we have
$$
\sum_{k=(p^r+3)/2}^{p^r-1}G(p^r,k)=\sum_{k=(p^r+3)/2}^{p^r-1}\frac{(-1)^{k+1}p^{2r}\binom{2p^r-1}{p^r-1}^2}{2(2p^r-1)4^{3(p^r-1)-k}}\binom{2p^r-2+2k}{2k}\frac{\binom{2p^r-1}{p^r-k}}{\binom{2k}k}.
$$
(\ref{2pr2k}) tells us that $p|\binom{2p^r-2+2k}{2k}$ for all $(p^r+3)/2\leq k<p^r$, and with (\ref{2ll2kk}) we have
\begin{align}\label{dao2kkhb}
\frac{-2p^r}{\binom{2k}k}\equiv(p^r-k)\binom{2p^r-2k}{p^r-k}\pmod{p}.
\end{align}
Hence
\begin{align*}
&\sum_{k=(p^r+3)/2}^{p^r-1}G(p^r,k)\\
&\equiv\frac{p^{r}\binom{2p^r-1}{p^r-1}^2}{4(2p^r-1)4^{3(p^r-1)}}\sum_{k=(p^r+3)/2}^{p^r-1}(p^r-k)\binom{2p^r-2+2k}{2k}\binom{2p^r-1}{p^r-k}\binom{2p^r-2k}{p^r-k}\\
&=\frac{p^{r}\binom{2p^r-1}{p^r-1}^2}{(2p^r-1)4^{3p^r-2}}\sum_{k=1}^{(p^r-3)/2}k(-4)^{p^r-k}\binom{4p^r-2-2k}{2p^r-2k}\binom{2p^r-1}{k}\binom{2k}{k}\\
&=\frac{p^{r}\binom{2p^r-1}{p^r-1}^2}{(2p^r-1)4^{3p^r-2}}\sum_{k=1}^{(p^r-3)/2}k(-4)^{p^r-k}\binom{-2p^r+1}{2p^r-2k}\binom{2p^r-1}{k}\binom{2k}{k}\\
&=-\frac{p^{2r}\binom{2p^r-1}{p^r-1}^2}{4^{2p^r-2}}\sum_{k=1}^{(p^r-3)/2}\frac{k\binom{2k}{k}}{(-4)^k}\frac{\binom{-2p^r-1}{2p^r-2k-2}\binom{2p^r-1}{k}}{(p^r-k)(2p^r-2k-1)}\pmod{p^{r+2}}.
\end{align*}
It is easy to see that for each $1\leq k\leq (p^r-3)/2$, we have
\begin{align}\label{2prfu1}
\binom{2p^r-1}{k}\equiv(-1)^k\pmod p
\end{align}
and
\begin{align}\label{fu2pr}
\binom{-2p^r-1}{2p^r-2k-2}=\prod_{i=1}^{2p^r-2k-2}\left(1+\frac{2p^r}{i}\right)=3\prod_{i=1,i\neq p^r}^{2p^r-2k-2}\left(1+\frac{2p^r}{i}\right)\equiv3\pmod p.
\end{align}
Thus
$$
\sum_{k=(p^r+3)/2}^{p^r-1}G(p^r,k)\equiv-3p^{2r}\sum_{k=1}^{(p^r-3)/2}\frac{k\binom{2k}{k}}{4^k}\frac{1}{(p^r-k)(2p^r-2k-1)}.
$$
We just need to see these items with $2k+1=p^{r-1}j, j\in\{1,2,\ldots,p-1\}$ and $2\nmid j$,
\begin{align*}
\sum_{k=(p^r+3)/2}^{p^r-1}G(p^r,k)&\equiv-3p^{r+1}\sum_{j=1,2\nmid j}^{p-1}\frac{\frac{p^{r-1}j-1}2\binom{p^{r-1}j-1}{\frac{p^{r-1}j-1}2}}{(p^r-\frac{p^{r-1}j-1}2)(2p-j)2^{p^{r-1}j-1}}\\
&\equiv-3p^{r+1}\sum_{j=1}^{(p-1)/2}\frac{\binom{2jp^{r-1}-p^{r-1}-1}{jp^{r-1}-\frac{p^{r-1}+1}2}}{(2j-1)4^{j-1}}\pmod{p^{r+2}}.
\end{align*}
It is easy to verify that
\begin{align}\label{2jpr}
\binom{2jp^{r-1}-p^{r-1}-1}{jp^{r-1}-\frac{p^r+1}2}&=(-1)^{jp^{r-1}+(p^{r-1}+1)/2}\prod_{i=1}^{jp^{r-1}-(p^{r-1}+1)/2}\left(1-\frac{(2j-1)p^{r-1}}i\right)\notag\\
&\equiv(-1)^{jp^{r-1}+(p^{r-1}+1)/2}\prod_{i=1}^{j-1}\left(1-\frac{2j-1}i\right)\notag\\
&=(-1)^{(p^{r-1}-1)/2}\binom{2j-2}{j-1}\pmod p.
\end{align}
So
$$
\sum_{k=(p^r+3)/2}^{p^r-1}G(p^r,k)\equiv-3p^{r+1}(-1)^{(p^{r-1}-1)/2}\sum_{j=0}^{(p-3)/2}\frac{\binom{2j}{j}}{(2j+1)4^{j}}\pmod{p^{r+2}}.
$$
This, with Lemma \ref{sun1} and $(-1)^{\frac{p^{r-1}-1}2}(-1)^{\frac{p-1}2}=(-1)^{\frac{p^{r}-1}2}$ yield the desired result.
\end{proof}
\noindent{\it Proof of Theorem \ref{Thgjmaa}}. Combining (\ref{wz1}) with Lemmas \ref{Fp1p1}--\ref{Gpr12pr}, we immediately get that
$$
\sum_{k=0}^{p^r-1}\frac{4k+1}{(-64)^k}\binom{2k}k^3\equiv(-1)^{(p^r-1)/2}p^r\pmod{p^{r+2}}.
$$
We know that $(-1)^{(p^r-1)/2}=(-1)^{\frac{(p-1)r}2}$. Therefore the proof of Theorem \ref{Thgjmaa} is complete.\qed
\section{Proof of Theorem \ref{Thglarxiv2}}
We just need to prove the theorem for $r>1$ since (\ref{glp4}) contains the theorem for $r=1$.

For nonnegative integer $n, k$, define
\begin{equation}\label{12F}
F(n,k)=(-1)^{n+k}\frac{(4n-1)\left(-\frac12\right)_n^2\left(-\frac12\right)_{n+k}}{(1)_n^2(1)_{n-k}\left(-\frac12\right)_k^2}
\end{equation}
and
\begin{equation}\label{12G}
G(n,k)=(-1)^{n+k}\frac{2\left(-\frac12\right)_n^2\left(-\frac12\right)_{n+k-1}}{(1)_{n-1}^2(1)_{n-k}\left(-\frac12\right)_k^2},
\end{equation}
where we assume that $1/(1)_n=0$ for $n=-1,-2,\ldots$. It can be easily verified that
\begin{equation}\label{FG1}
F(n,k-1)-F(n,k)=G(n+1,k)-G(n,k)
\end{equation}
for all nonnegative integer $n$ and $k>0$.

We know
$$
\sum_{n=0}^{p^r-1}(-1)^n(4n-1)\frac{\left(-\frac12\right)_n^3}{(1)_n^3}=\sum_{n=0}^{p^r-1}F(n,0).
$$
Summing (\ref{FG1}) over $n$ from $0$ to $p^r-1$ and then over $k$ from $1$ to $p^r-1$, we have
\begin{align}\label{wz3}
\sum_{n=0}^{p^r-1}F(n,0)=F(p^r-1,p^r-1)+\sum_{k=1}^{p^r-1}G(p^r,k).
\end{align}
\begin{lem}\label{F2prpr}
$$
F(p^r-1,p^r-1)\equiv0\pmod{p^{r+2}}.
$$
\end{lem}
\begin{proof}
By (\ref{12F}), we have
\begin{align*}
F(p^r-1,p^r-1)=\frac{(4p^r-5)\left(-\frac12\right)_{2p^r-2}}{(1)^2_{p^r-1}}=-\frac{4p^r-5}{4p^r-4}\frac{\binom{4p^r-6}{2p^r-3}\binom{2p^r-2}{p^r-1}}{4^{2p^r-3}}\equiv0\pmod{p^{r+2}},
\end{align*}
since $\binom{4p^r-6}{2p^r-3}\equiv\binom{2p^r-2}{p^r-1}\equiv0\pmod{p^r}$ and $r>1$.
\end{proof}
It is easy to see from (\ref{12G}) that
\begin{align*}
G(p^r,k)&=(-1)^{k+1}\frac{\left(-\frac12\right)^2_{p^r}\left(-\frac12\right)_{p^r-1+k}}{(1)^2_{p^r-1}(1)_{p^r-k}\left(-\frac12\right)^2_{k}}=\frac{(-1)^{k+1}}2\frac{\binom{2p^r-2}{p^r-1}^2}{4^{2p^r-2}}\frac{\left(-\frac12\right)_{p^r-1+k}}{(1)_{p^r-k}\left(-\frac12\right)^2_{k}}\\
&=\frac{\binom{2p^r-2}{p^r-1}^2}{4^{3p^r-2}}\frac{(-4)^k}{\binom{2k-2}{k-1}}\binom{2p^r-4+2k}{2k-2}\binom{2p^r-2}{p^r-k}.
\end{align*}
Hence
\begin{align*}
\sum_{k=1}^{p^r-1}G(p^r,k)&=\frac{\binom{2p^r-2}{p^r-1}^2}{4^{3p^r-2}}\sum_{k=1}^{p^r-1}\frac{(-4)^k}{\binom{2k-2}{k-1}}\binom{2p^r-4+2k}{2k-2}\binom{2p^r-2}{p^r-k}\\
&=\frac{\binom{2p^r-2}{p^r-1}^2}{4^{3p^r-2}}\sum_{k=0}^{p^r-2}\frac{(-4)^{k+1}}{\binom{2k}{k}}\binom{2p^r-2+2k}{2k}\binom{2p^r-2}{p^r-k-1}.
\end{align*}
Since
$$
\frac{\binom{2p^r-2}{p^r-1}^2}{4^{3p^r-2}}(-4)\binom{2p^r-2}{p^r-1}\equiv0\pmod{p^{r+2}}
$$
and
$$
\frac{\binom{2p^r-2}{p^r-1}^2}{4^{3p^r-2}}\frac{(-4)^{p^r-1+1}}{\binom{2p^r-2}{p^r-1}}\binom{4p^r-4}{2p^r-2}\equiv0\pmod{p^{r+2}}.
$$
Then we have
\begin{align}\label{prGpr}
\sum_{k=1}^{p^r-1}G(p^r,k)&\equiv\frac{p^{2r}\binom{2p^r-1}{p^r-1}^2}{(2p^r-1)^24^{3p^r-2}}\sum_{k=1}^{p^r-1}\frac{(-4)^{k+1}}{\binom{2k}{k}}\binom{-2p^r+1}{2k}\binom{2p^r-2}{p^r-k-1}\notag\\
&=-\frac{p^{3r}\binom{2p^r-1}{p^r-1}^2}{(2p^r-1)4^{3p^r-3}}\sum_{k=1}^{p^r-1}\frac{(-4)^{k}}{\binom{2k}{k}}\frac{\binom{-2p^r-1}{2k-2}}{k(2k-1)}\binom{2p^r-2}{p^r-k-1}\pmod{p^{r+2}}.
\end{align}
\begin{lem}\label{theta1}
$$
\theta_1=-\frac{p^{3r}\binom{2p^r-1}{p^r-1}^2}{(2p^r-1)4^{3p^r-3}}\sum_{k=1}^{(p^r-1)/2}\frac{(-4)^{k}}{\binom{2k}{k}}\frac{\binom{-2p^r-1}{2k-2}}{k(2k-1)}\binom{2p^r-2}{p^r-k-1}\equiv0\pmod{p^{r+2}}.
$$
\end{lem}
\begin{proof}
By (\ref{2kk}) and (\ref{dao2kk}) we have
$$
\frac{p^r}{k(2k-1)}\equiv\frac{p^r}{\binom{2k}k}\equiv0\pmod p
$$
for each $1\leq k\leq(p^r-1)/2$. So we immediately obtain the desired result.
\end{proof}
\begin{lem}\label{theta2}
\begin{align*}
\theta_2&=-\frac{p^{3r}\binom{2p^r-1}{p^r-1}^2}{(2p^r-1)4^{3p^r-3}}\frac{(-4)^{(p^r+1)/2}}{\binom{p^r+1}{(p^r+1)/2}}\frac{\binom{-2p^r-1}{p^r-1}}{(p^r+1)/2(p^r)}\binom{2p^r-2}{(p^r-3)/2}\\
&\equiv-(-1)^{(p^r-1)/2}p^r(1-3pq_p(2))\pmod{p^{r+2}}.
\end{align*}
\end{lem}
\begin{proof}
By simple calculation, we have
\begin{align*}
\theta_2=\frac{(-1)^{(p^r-1)/2}\binom{2p^r-1}{p^r-1}^2(p^r-1)p^r}{(2p^r-1)^24^{3p^r-2-(p^r+1)/2}}\frac{\binom{-2p^r-1}{p^r-1}\binom{2p^r-1}{(p^r-1)/2}}{\binom{p^r-1}{(p^r-1)/2}}
\end{align*}
In the same way of computing (\ref{mao}), we can deduce that
$$
\binom{-2p^r-1}{p^r-1}\equiv1\pmod{p^2}.
$$
This, with (\ref{pr12pr}), (\ref{2prprp2}) and (\ref{2pr12pr}) yield that
\begin{align*}
\theta_2&\equiv(-1)^{(p^r-1)/2}p^r\frac{p^r-1}{(2p^r-1)^22^{7(p^r-1)}}(1+4pq_p(2))\\
&\equiv-(-1)^{(p^r-1)/2}p^r\frac{1+4pq_p(2)}{1+7pq_p(2)}\equiv-(-1)^{(p^r-1)/2}p^r(1-3pq_p(2))\pmod{p^{r+2}},
\end{align*}
where we used that $2^{p^r-1}=(1+pq_p(2))^{(p^r-1)/(p-1)}\equiv1+pq_p(2)\pmod{p^2}$.
\end{proof}
\begin{lem}\label{theta3}
\begin{align*}
\theta_3&=-\frac{p^{3r}\binom{2p^r-1}{p^r-1}^2}{(2p^r-1)4^{3p^r-3}}\sum_{k=(p^r+3)/2}^{p^r-1}\frac{(-4)^{k}}{\binom{2k}{k}}\frac{\binom{-2p^r-1}{2k-2}}{k(2k-1)}\binom{2p^r-2}{p^r-k-1}
\\&\equiv-(-1)^{(p^r-1)/2}3p^{r+1}q_p(2)\pmod{p^{r+2}}.
\end{align*}
\end{lem}
\begin{proof}
Note that $\frac{p^r}{k(2k-1)}\equiv0\pmod p$ for all $(p^r+3)/2\leq k\leq p^r-1$, so with (\ref{2ll2kk}) and Fermat's little theorem, we have
\begin{align*}
\theta_3&\equiv\frac{p^{2r}\binom{2p^r-1}{p^r-1}^2}{2(2p^r-1)4^{3p^r-3}}\sum_{k=(p^r+3)/2}^{p^r-1}(-4)^k(p^r-k)\binom{2p^r-2k}{p^r-k}\frac{\binom{-2p^r-1}{2k-2}}{k(2k-1)}\binom{2p^r-2}{p^r-k-1}\\
&\equiv \frac{p^{2r}}2\sum_{k=1}^{(p^r-3)/2}\frac{(-4)^{p^r-k}\binom{2k}{k}}{2p^r-2k-1}\binom{-2p^r-1}{2p^r-2k-2}\binom{2p^r-2}{k-1}\\
&\equiv-2p^{2r}\sum_{k=1}^{(p^r-3)/2}\frac{k\binom{2k}k}{(2k+1)(-4)^k}\binom{-2p^r-1}{2p^r-2k-2}\binom{2p^r-1}{k}\pmod{p^{r+2}}.
\end{align*}
This, with (\ref{2prfu1}) and (\ref{fu2pr}) yield that
$$
\theta_3\equiv-6p^{2r}\sum_{k=1}^{(p^r-3)/2}\frac{k\binom{2k}k}{(2k+1)(4)^k}\pmod{p^{r+2}}.
$$
We just need to compute these items with $2k+1=p^{r-1}j, 2\nmid j$. So modulo $p^{r+2}$ we have
$$
\theta_3\equiv-6p^{2r}\sum_{j=1,2\nmid j}^{p-1}\frac{\frac{p^{r-1}j-1}2\binom{p^{r-1}j-1}{\frac{p^{r-1}j-1}2}}{2^{p^{r-1}j-1}p^{r-1}j}\equiv3p^{r+1}\sum_{j=1,2\nmid j}^{p-1}\frac{\binom{p^{r-1}j-1}{\frac{p^{r-1}j-1}2}}{j2^{j-1}}=3p^{r+1}\sum_{j=0}^{(p-3)/2}\frac{\binom{2jp^{r-1}+p^{r-1}-1}{jp^{r-1}+(p^{r-1}-1)/2}}{(2j+1)4^{j}}.
$$
Hence with (\ref{2jpr}), we immediately get that
$$
\theta_3\equiv(-1)^{(p^{r-1}-1)/2}3p^{r+1}\sum_{j=0}^{(p-3)/2}\frac{\binom{2j}j}{(2j+1)4^{j}}\pmod{p^{r+2}}.
$$
This, with Lemma \ref{sun1} and $(-1)^{\frac{p^{r-1}-1}2}(-1)^{\frac{p-1}2}=(-1)^{\frac{p^{r}-1}2}$ yield the desired result.
\end{proof}
\noindent{\it Proof of Theorem \ref{Thglarxiv2}}. Substituting Lemmas \ref{theta1}--\ref{theta3} into (\ref{prGpr}), we have
$$
\sum_{k=1}^{p^r-1}G(p^r,k)\equiv-(-1)^{(p^r-1)/2}p^r\pmod{p^{r+2}}.
$$
This, with Lemma \ref{F2prpr} and (\ref{wz3}) yield that
$$
\sum_{n=0}^{p^r-1}F(n,0)\equiv-(-1)^{(p^r-1)/2}p^r\pmod{p^{r+2}}.
$$
Now we finish the proof of Theorem \ref{Thglarxiv2}.\qed

\section{Proof of Theorem \ref{Th20n3}}
We will use the following WZ pair which can be found in \cite{zudilin-jnt-2009} or \cite{sun-ijm-2012} to prove Theorem \ref{Th20n3}. For nonnegative integers $n, k$, define
$$
F(n,k)=\frac{(-1)^{n+k}(20n-2k+3)}{4^{5n-k}}\binom{2n}n\frac{\binom{4n+2k}{2n+k}\binom{2n+k}{2k}\binom{2n-k}{n}}{\binom{2k}k}
$$
and
$$
G(n,k)=\frac{(-1)^{n+k}}{4^{5n-4-k}}\frac{n\binom{2n-1}{n-1}\binom{2(2n-1+k)}{2n-1+k}\binom{2n-1+k}{2k}\binom{2n-1-k}{n-1}}{\binom{2k}k}.
$$
Clearly $F(n,k)=G(n,k)=0$ if $n<k$. It is easy to check that
\begin{equation}\label{FG5}
F(n,k-1)-F(n,k)=G(n+1,k)-G(n,k)
\end{equation}
for all nonnegative integer $n$ and $k>0$.

We mentioned that Zudilin has proved the theorem for $r=1$, so we just need to show that for $r>1$.

Summing (\ref{FG5}) over $n$ from $0$ to $p^r-1$ we have
$$
\sum_{n=0}^{p^r-1}F(n,k-1)-\sum_{n=0}^{p^r-1}F(n,k)=G(p^r,k)-G(0,k)=G(p^r,k).
$$
Furthermore, summing both side of the above identity over $k$ from $1$ to $p^r-1$, we obtain
\begin{align}\label{wz5}
\sum_{n=0}^{p^r-1}F(n,0)=F(p^r-1,p^r-1)+\sum_{k=1}^{p^r-1}G(p^r,k).
\end{align}
\begin{lem} \label{Fp1p13}
$$F(p^r-1,p^r-1)\equiv 0\pmod{p^{r+2}}.$$
\end{lem}
\begin{proof} Since $r>1$, we have
\begin{align*}
F(p^r-1,p^r-1)&=\frac{3(6p^r-5)}{4^{4p^r-4}}\binom{6p^r-6}{3p^r-3}\binom{3p^r-3}{p^r-1}=\frac{3p^{r}}{4^{4p^r-4}}\binom{6p^r-5}{3p^r-3}\binom{3p^r-2}{2p^r-2}\\
&=\frac{3p^{r}}{4^{4p^r-3}}\binom{6p^r-3}{3p^r-1}\binom{3p^r-1}{2p^r-1}\\&=\frac{9p^{2r}}{2\cdot4^{4p^r-3}}\binom{6p^r-1}{3p^r-1}\binom{3p^r-1}{2p^r-1}\frac1{6p^r-1}\equiv0\pmod{p^{r+2}}.
\end{align*}
\end{proof}
By the definition of $G(n,k)$ we have
\begin{align}\label{Gpk3}
G(p^r,k)&=\frac{(-1)^{k+1}p^r\binom{2p^r-1}{p^r-1}}{4^{5p^r-4-k}}\binom{4p^r-2+2k}{2p^r-1+k}\frac{\binom{2p^r-1+k}{2k}\binom{2p^r-1-k}{p^r-1}}{\binom{2k}k}\notag\\
&=\frac{(-1)^{k+1}p^r\binom{2p^r-1}{p^r-1}}{4^{5p^r-4-k}}\binom{4p^r-2+2k}{2k}\frac{\binom{4p^r-2}{2p^r-k-1}\binom{2p^r-1-k}{p^r-1}}{\binom{2k}k}\notag\\
&=\frac{(-1)^{k+1}p^r\binom{2p^r-1}{p^r-1}}{4^{5p^r-4-k}}\binom{4p^r-2+2k}{2k}\frac{\binom{4p^r-2}{3p^r-1}\binom{3p^r-1}{p^r-k}}{\binom{2k}k}\notag\\
&=\frac{(-1)^{k+1}\binom{2p^r-1}{p^r-1}\binom{4p^r-1}{p^r-1}}{4^{5p^r-4-k}}\frac{6p^{3r}}{k(2k-1)}\frac{\binom{-4p^r-1}{2k-2}\binom{3p^r-1}{p^r-k}}{\binom{2k}k},
\end{align}
where we used the binomial transformation
$$
\binom{n}{k}\binom{k}{j}=\binom{n}{j}\binom{n-j}{k-j}\ \mbox{and}\ \binom{n}{k}=(-1)^k\binom{-n+k-1}{k}.
$$
\begin{lem} \label{G123}
$$
\sum_{k=1}^{(p^r-1)/2}G(p^r,k)\equiv 0\pmod{p^{r+2}}.
$$
\end{lem}
\begin{proof}
By (\ref{Gpk3}), we have
\begin{align*}
\sum_{k=1}^{(p^r-1)/2}G(p^r,k)=-\frac{\binom{2p^r-1}{p^r-1}\binom{4p^r-1}{p^r-1}}{4^{5p^r-4}}\sum_{k=1}^{(p^r-1)/2}\frac{6p^{3r}(-4)^k}{k(2k-1)}\frac{\binom{-4p^r-1}{2k-2}\binom{3p^r-1}{p^r-k}}{\binom{2k}k}.
\end{align*}
It is easy to see that $p^r/(k(2k-1))\equiv0\pmod p$, this with (\ref{2kk}) and (\ref{dao2kk}) yield that
$$\sum_{k=1}^{(p^r-1)/2}G(p^r,k)\equiv 0\pmod{p^{r+2}}.$$
\end{proof}
\begin{lem}\label{Gpr3}
$$
G(p^r,(p^r+1)/2)\equiv(-1)^{(p^r-1)/2}3p^r(1-5pq_p(2))\pmod{p^{r+2}},
$$
where $q_p(2)=(2^{p-1}-1)/p$ stands for the Fermat quotient.
\end{lem}
\begin{proof}
By (\ref{Gpk3}) and (\ref{pr12pr})--(\ref{h12p}) we have
\begin{align*}
G(p^r,(p^r+1)/2)&=\frac{(-1)^{(p^r-1)/2}12p^{2r}}{4^{5p^r-4-(p^r+1)/2}}\frac{\binom{2p^r-1}{p^r-1}\binom{4p^r-1}{p^r-1}\binom{-4p^r-1}{p^r-1}\binom{3p^r-1}{(p^r-1)/2}}{(p^r+1)\binom{p^r+1}{(p^r+1)/2}}\\
&\equiv\frac{(-1)^{(p^r-1)/2}3p^{r}}{4^{5p^r-4-(p^r+1)/2}}\frac{(-1)^{(p^r-1)/2}(1+6pq_p(2))}{(-1)^{(p^r-1)/2}4^{p^r-1}}\pmod{p^{r+2}}.
\end{align*}
Hence
$$
G(p^r,(p^r+1)/2)\equiv\frac{(-1)^{(p^r-1)/2}3p^{r}(1+6pq_p(2))}{2^{11(p^r-1)}}\pmod{p^{r+2}}.
$$
Therefore the desired result immediately obtained since
$$2^{p^r-1}\equiv1+pq_p(2)\pmod{p^2}.$$
So the proof of Lemma \ref{Gpr3} is finished.
\end{proof}
\begin{lem}\label{Gpr12pr3}
$$\sum_{k=(p^r+3)/2}^{p^r-1}G(p^r,k)\equiv(-1)^{(p^r-1)/2}15p^{r+1}q_p(2)\pmod{p^{r+2}}.$$
\end{lem}
\begin{proof}
In view of (\ref{Gpk3}), we have
$$
\sum_{k=(p^r+3)/2}^{p^r-1}G(p^r,k)=-\frac{\binom{2p^r-1}{p^r-1}\binom{4p^r-1}{p^r-1}}{4^{5p^r-4}}\sum_{k=(p^r+3)/2}^{p^r-1}\frac{6p^{3r}(-4)^k}{k(2k-1)}\frac{\binom{-4p^r-1}{2k-2}\binom{3p^r-1}{p^r-k}}{\binom{2k}k}.
$$
It is easy to see that $p^r/(k(2k-1))\equiv0\pmod p$ for all $(p^r+3)/2\leq k<p^r$, and with (\ref{dao2kk}) we have
\begin{align*}
&\sum_{k=(p^r+3)/2}^{p^r-1}G(p^r,k)\\
&\equiv\frac{\binom{2p^r-1}{p^r-1}\binom{4p^r-1}{p^r-1}}{4^{5p^r-4}}\sum_{k=(p^r+3)/2}^{p^r-1}\frac{3p^{2r}(-4)^k}{k(2k-1)}\binom{-4p^r-1}{2k-2}\binom{3p^r-1}{p^r-k}\binom{2p^r-2k}{p^r-k}(p^r-k)\\
&\equiv-3p^r\sum_{k=1}^{(p^r-3)/2}\frac{p^{r}\binom{2k}k}{(-4)^k(2k+1)}\binom{-4p^r-1}{2p^r-2k-2}\binom{3p^r-1}{k}\pmod{p^{r+2}}.
\end{align*}
It is easy to see that for each $1\leq k\leq (p^r-3)/2$, we have
\begin{align}\label{3prfu1}
\binom{3p^r-1}{k}\equiv(-1)^k\pmod p
\end{align}
and
\begin{align}\label{fu4pr}
\binom{-4p^r-1}{2p^r-2k-2}=\prod_{i=1}^{2p^r-2k-2}\left(1+\frac{4p^r}{i}\right)=5\prod_{i=1,i\neq p^r}^{2p^r-2k-2}\left(1+\frac{2p^r}{i}\right)\equiv5\pmod p.
\end{align}
Thus
$$
\sum_{k=(p^r+3)/2}^{p^r-1}G(p^r,k)\equiv-15p^{2r}\sum_{k=1}^{(p^r-3)/2}\frac{\binom{2k}{k}}{4^k(2k+1)}\pmod{p^{r+2}}.
$$
As the same way of proving Lemma \ref{Gpr12pr}, we have
$$
\sum_{k=(p^r+3)/2}^{p^r-1}G(p^r,k)\equiv-15p^{r+1}(-1)^{(p^{r-1}-1)/2}\sum_{j=0}^{(p-3)/2}\frac{\binom{2j}{j}}{(2j+1)4^{j}}\pmod{p^{r+2}}.
$$
This, with Lemma \ref{sun1} and $(-1)^{\frac{p^{r-1}-1}2}(-1)^{\frac{p-1}2}=(-1)^{\frac{p^{r}-1}2}$ yield the desired result.
\end{proof}
\noindent{\it Proof of Theorem \ref{Th20n3}}. Combining (\ref{wz5}) with Lemmas \ref{Fp1p13}--\ref{Gpr12pr3}, we immediately get that
$$
\sum_{n=0}^{p^r-1}\frac{\left(\frac12\right)_n\left(\frac12\right)_{2n}}{n!^3}(20n+3)\frac{1}{2^{4n}}\equiv(-1)^{(p^r-1)/2}3p^r\pmod{p^{r+2}}.
$$
Therefore the proof of Theorem \ref{Th20n3} is complete.\qed

\section{Proof of Theorem \ref{swisher}}
\noindent{\it Proof of (\ref{gvan})}. By the identity in \cite[Lemma 2.2]{mao-rama-2018}, we obtain that
$$
\sum_{n=0}^{(p-1)/2}\frac{\binom{2n}n^2}{(n+1)16^n}=\sum_{n=0}^{(p-1)/2}\frac{\binom{-1/2}n^2}{n+1}=\frac{\binom{-3/2}{(p-1)/2}^2}{(p-1)/2+1}.
$$
With direct computation, we have
$$
\frac{\binom{-3/2}{(p-1)/2}^2}{(p-1)/2+1}=\frac{2p^{2}}{(p^{2}-1)(p-1)}\binom{p/2-1}{(p-3)/2}^2\equiv2p^{2}(1+p)(1-pH_{(p-3)/2})\pmod{p^{4}},
$$
since
$$\binom{p/2-1}{(p-3)/2}^2=\prod_{k=1}^{(p-3)/2}\left(1-\frac{p}{2k}\right)^2\equiv1-pH_{(p-3)/2}\pmod {p^2}.$$
We immediately obtain the desired result with (\ref{h12p}).\qed

\vskip 3mm \noindent{\bf Acknowledgments.}
The author is funded by the Startup Foundation for Introducing Talent of Nanjing University of Information Science and Technology (2019r062).

\end{document}